\documentclass{amsart}
\usepackage{amsfonts}

\newtheorem{theorem}{Theorem}
\theoremstyle{plain}

\newtheorem{lemma}{Lemma}

\numberwithin{equation}{section}
\input{tcilatex}

\begin{document}
\title[On an infinite series for $(1+1/x)^{x}$]{On an infinite series for $%
(1+1/x)^{x}$}
\author{Cristinel Mortici}
\address{Valahia University of T\^{a}rgovi\c{s}te, Department of
Mathematics, Bd. Unirii 18, 130082 T\^{a}rgovi\c{s}te, Romania}
\email{cmortici@valahia.ro}
\author{Yue Hu}
\address{College of Mathematics and Informatics, Henan Polytechnic
University, \\
Jaozuo City, Henan 454000, P.R. China. }
\email{huu3y2@163.com}

\begin{abstract}
The aim of this paper is to construct a new expansion of $\left(
1+1/x\right) ^{x}$ related to Carleman's inequality. Our results extend some
results of Yang [Approximations for constant $e$ and their applications J.
Math. Anal. Appl. 262 (2001) 651-659].
\end{abstract}

\maketitle

\section{Introduction}

The following Carleman inequality \cite{carl}%
\begin{equation*}
\sum_{n=1}^{\infty }\left( a_{1}a_{2}\cdots a_{n}\right)
^{1/n}<e\sum_{n=1}^{\infty }a_{n},
\end{equation*}%
whenever $a_{n}\geq 0,$ $n=1,2,3,\ldots ,$ with $0<\sum_{n=1}^{\infty
}a_{n}<\infty ,$ has attracted the attention of many authors in the recent
past. We refer here to the following results%
\begin{equation*}
\sum_{n=1}^{\infty }\left( a_{1}a_{2}\cdots a_{n}\right)
^{1/n}<e\sum_{n=1}^{\infty }\left( 1-\frac{1}{2n+2}\right) a_{n}\ ,\ \ \ \ (%
\text{Bicheng and Debnath \cite{bicheng}})
\end{equation*}%
\begin{equation*}
\sum_{n=1}^{\infty }\left( a_{1}a_{2}\cdots a_{n}\right)
^{1/n}<e\sum_{n=1}^{\infty }\left( 1+\frac{1}{n+1/5}\right) ^{-1/2}a_{n}\ ,\
\ \ \ (\text{Ping and Guozheng \cite{ping}})
\end{equation*}%
\begin{equation*}
\sum_{n=1}^{\infty }\left( a_{1}a_{2}\cdots a_{n}\right)
^{1/n}<e\sum_{n=1}^{\infty }\left( 1-\frac{1}{2cn+4c/3+1/2 }%
\right)^{c}a_{n}\ ,\ \ \ \ (\text{Yang \cite{yang}}).
\end{equation*}%
Moreover, Yang \cite{yang} proved%
\begin{equation*}
\sum_{n=1}^{\infty }\left( a_{1}a_{2}\cdots a_{n}\right)
^{1/n}<e\sum_{n=1}^{\infty }\left( 1-\sum_{k=1}^{6}\frac{b_{k}}{\left(
n+1\right) ^{k}}\right) a_{n}
\end{equation*}%
with $b_{1}=1/2,$ $b_{2}=1/24,$ $b_{3}=1/48,$ $b_{4}=73/5760$, $%
b_{5}=11/1280,$ $b_{6}=1945/580608,$ then conjectured that if%
\begin{equation}
\left( 1+\frac{1}{x}\right) ^{x}=e\left( 1-\sum_{k=1}^{\infty }\frac{b_{k}}{%
\left( x+1\right) ^{k}}\right) ,\ \ \ x>0,  \label{S0}
\end{equation}%
then $b_{k}>0,$ $k=1,2,3,\ldots $ .

This open problem was recently solved by Yang \cite{yang}, who proved%
\begin{equation*}
\sum_{n=1}^{\infty }\left( a_{1}a_{2}\cdots a_{n}\right)
^{1/n}<e\sum_{n=1}^{\infty }\left( 1-\sum_{k=1}^{\infty }\frac{b_{k}}{\left(
n+1\right) ^{k}}\right) a_{n},
\end{equation*}%
whenever $a_{n}\geq 0,$ $n=1,2,3,\ldots ,$ with $0<\sum_{n=1}^{\infty
}a_{n}<\infty ,$ where $b_{0}=1$ and%
\begin{equation*}
b_{n}=\frac{1}{n}\left( -\sum_{k=0}^{n-2}\frac{b_{n-1-k}}{k+1}+\frac{1}{n+1}%
\right) .
\end{equation*}%
This conjecture was proved and discussed also by Yang \cite{yang2},
Gyllenberg and Yan \cite{gy}, and Yue \cite{yue}. In the final part of his
paper, Yang \cite{yang2} remarked that in order to obtain better results,
the right-hand side of (\ref{S0}) could be replaced by $e\left[
1-\sum_{n=1}^{\infty }\left( d_{n}/\left( x+\varepsilon \right) ^{n}\right) %
\right] ,$ where $\varepsilon \in (0,1]$ and $d_{n}=d_{n}\left( \varepsilon
\right) ,$ but informations about values of $\varepsilon $ are not provided.
We prove in this paper that $\varepsilon =11/12$ provides the fastes series $%
\sum_{n=1}^{\infty }\left( d_{n}/\left( x+\varepsilon \right) ^{n}\right) $
and also formulas for coefficients $d_{n}$ are given.

\section{The Results}

By truncation of the series%
\begin{equation}
\left( 1+\frac{1}{n}\right) ^{n}=e\left( 1-\frac{b_{1}}{n+1}-\frac{b_{2}}{%
\left( n+1\right) ^{2}}-...\right)  \label{S}
\end{equation}%
we obtain approximations of any desired accuracy $n^{-k}$. The first
approximation is%
\begin{equation}
\left( 1+\frac{1}{n}\right) ^{n}\approx e\left( 1-\frac{\frac{1}{2}}{n+1}%
\right)  \label{0}
\end{equation}%
but it is interesting to find the best approximation of the form%
\begin{equation*}
\left( 1+\frac{1}{n}\right) ^{n}\approx e\left( 1-\frac{a}{n+b}\right) ,%
\text{ \ \ as }n\rightarrow \infty .
\end{equation*}%
This problem was solved in \cite{m1}, where the best values $a=2,$ $b=\frac{%
11}{6}$ were found. The proof of this fact is based on the following lemma,
which is a powerfull tool for measuring the speed of convergence.

\begin{lemma}
\emph{If }$\left( \omega _{n}\right) _{n\geq 1}$ \emph{is convergent to zero}
\emph{and there exists the limit}%
\begin{equation*}
\lim_{n\rightarrow \infty }n^{k}(\omega _{n}-\omega _{n+1})=l\in 
\mathbb{R}
,
\end{equation*}%
\emph{with }$k>1,$ \emph{then there exists the limit:}%
\begin{equation*}
\lim_{n\rightarrow \infty }n^{k-1}\omega _{n}=\frac{l}{k-1}.
\end{equation*}
\end{lemma}

Hence if replace $n+1$ by $n+\frac{11}{12}$ in (\ref{0}), a better
approximation can be obtained. An idea arises naturally: to construct a
series (\ref{S}) in negative powers of $n+\frac{11}{12}.$

This fact will also solve an open problem posed by Yang, who remarked in the
final part of his paper \cite{yang2} that in order to obtain better results,
the right side of (\ref{S}) could be replaced by $e\left[ 1-\sum_{n=1}^{%
\infty }\left( d_{n}/\left( x+\varepsilon \right) ^{n}\right) \right] ,$
where $\varepsilon \in (0,1],$ but informations about values of $\varepsilon 
$ are not provided.

Our above studies show that the value $\varepsilon =\frac{11}{12}$ gives
indeed better results. The same method using Lemma 1 produces the series%
\begin{equation}
\left( 1+\frac{1}{n}\right) ^{n}=e\left( 1-\frac{\frac{1}{2}}{n+\frac{11}{12}%
}-\frac{\frac{5}{288}}{\left( n+\frac{11}{12}\right) ^{3}}-\frac{\frac{139}{%
17\,280}}{\left( n+\frac{11}{12}\right) ^{4}}-\frac{\frac{119}{23\,040}}{%
\left( n+\frac{11}{12}\right) ^{5}}-\cdots \right)  \label{D}
\end{equation}%
which is better than (\ref{S}), since by truncation after $k\geq 3$ terms of
series (\ref{S}), the last term is of order $n^{-\left( k-1\right) },$ while
the last term of series (\ref{D}) truncated after $k$ terms is of order $%
n^{-k}.$

In order to obtain the next coefficient of $\left( n+\frac{11}{12}\right)
^{-2}$ in (\ref{D}), we search the best approximation%
\begin{equation*}
\left( 1+\frac{1}{n}\right) ^{n}\approx e\left( 1-\frac{\frac{1}{2}}{\left(
n+\frac{11}{12}\right) }-\frac{c}{\left( n+\frac{11}{12}\right) ^{2}}\right)
,\ \ \ \text{as \ }n\rightarrow \infty .
\end{equation*}%
Such an approximation is better as the relative error sequence defined by%
\begin{equation*}
\left( 1+\frac{1}{n}\right) ^{n}=e\left( 1-\frac{\frac{1}{2}}{\left( n+\frac{%
11}{12}\right) }-\frac{c}{\left( n+\frac{11}{12}\right) ^{2}}\right) \exp
\omega _{n}\ ,\ \ \ n\geq 1,
\end{equation*}%
converges faster to zero. By using computer algebra, we get%
\begin{equation*}
\omega _{n}-\omega _{n+1}=\frac{2c}{n^{3}}-\left( 7c+\frac{5}{96}\right) 
\frac{1}{n^{4}}+O\left( \frac{1}{n^{5}}\right) .
\end{equation*}%
According to Lemma 1, the fastest sequence $\omega _{n}$ is obtained for $%
c=0.$

With $c=0,$ let us now search the best approximation of the form%
\begin{equation*}
\left( 1+\frac{1}{n}\right) ^{n}\approx e\left( 1-\frac{\frac{1}{2}}{\left(
n+\frac{11}{12}\right) }-\frac{d}{\left( n+\frac{11}{12}\right) ^{3}}\right)
,\ \ \ \text{as \ }n\rightarrow \infty .
\end{equation*}%
For the corresponding relative error sequence $w_{n}$ given by%
\begin{equation*}
\left( 1+\frac{1}{n}\right) ^{n}=e\left( 1-\frac{\frac{1}{2}}{\left( n+\frac{%
11}{12}\right) }-\frac{d}{\left( n+\frac{11}{12}\right) ^{3}}\right) \exp
w_{n}\ ,\ \ \ n\geq 1,
\end{equation*}%
we have%
\begin{equation*}
w_{n}-w_{n+1}=\left( 3d-\frac{5}{96}\right) \frac{1}{n^{4}}+\left( -15d+%
\frac{493}{2160}\right) \frac{1}{n^{5}}+O\left( \frac{1}{n^{6}}\right) .
\end{equation*}%
The fastest sequence $w_{n}$ is obtained when the coefficient of $n^{-4}$
vanishes, that is $d=\frac{5}{288}.$ More coefficients in (\ref{D}) can be
inductively obtained.

\section{The general term of $d_{n}$}

Now it is natural to ask the general term, or at least a recurrence relation
of $d_{n}$ in (\ref{D}), that is%
\begin{equation}
\left( 1+\frac{1}{n}\right) ^{n}=e\left( 1-\sum_{k=1}^{\infty }\frac{d_{k}}{%
\left( n+\frac{11}{12}\right) ^{k}}\right) .  \label{DD}
\end{equation}

By (\ref{D}), we have $d_{1}=\frac{1}{2},$ $d_{2}=0,$ $d_{3}=\frac{5}{288},$ 
$d_{4}=\frac{139}{17\,280},$ $d_{5}=\frac{119}{23\,040},$ $...$ .

One idea for the complete characterization of $d_{n}$ is to provide a
formula in term of $b_{k},$ as we can see in the following

\begin{theorem}
Let $b_{0}=1$ and $b_{n}=\frac{1}{n}\left( -\sum_{k=0}^{n-2}\frac{b_{n-1-k}}{%
k+1}+\frac{1}{n+1}\right) ,$ $n\geq 2.$ Then if%
\begin{equation*}
\left( 1+\frac{1}{m}\right) ^{m}=e\left( d_{0}-\frac{d_{1}}{m+\frac{11}{12}}-%
\frac{d_{2}}{\left( m+\frac{11}{12}\right) ^{2}}-\cdots \right) ,
\end{equation*}%
then $d_{0}=1$ and%
\begin{equation*}
d_{s}=\Gamma \left( s\right) \sum_{k=1}^{s}\left( -1\right) ^{s-k}\frac{b_{k}%
}{\Gamma \left( s-k+1\right) \Gamma \left( k\right) 12^{s-k}},\ \ \
s=1,2,3,\ldots .
\end{equation*}
\end{theorem}

\begin{proof}
First by the binomial formula, we have%
\begin{equation*}
\left( 1+\frac{1}{12}t\right) ^{-k}=\sum_{n=0}^{\infty }\left( -1\right) ^{n}%
\frac{\Gamma \left( n+k\right) }{\Gamma \left( k\right) \Gamma \left(
n+1\right) 12^{n}}t^{n},
\end{equation*}%
and with $t=\left( m+\frac{11}{12}\right) ^{-1},$%
\begin{equation*}
\left( 1-\frac{\frac{1}{12}}{m+1}\right) ^{k}=\sum_{n=0}^{\infty }\left(
-1\right) ^{n}\frac{\Gamma \left( n+k\right) }{\Gamma \left( k\right) \Gamma
\left( n+1\right) 12^{n}}\frac{1}{\left( m+\frac{11}{12}\right) ^{n}}.
\end{equation*}%
Now%
\begin{eqnarray*}
\sum_{k=0}^{\infty }\frac{b_{k}}{\left( m+1\right) ^{k}} &=&\sum_{k=0}^{%
\infty }\left( 1-\frac{\frac{1}{12}}{m+1}\right) ^{k}\frac{b_{k}}{\left( m+%
\frac{11}{12}\right) ^{k}} \\
&=&\sum_{k=0}^{\infty }\sum_{n=0}^{\infty }\left( -1\right) ^{n}\frac{\Gamma
\left( k+n\right) }{\Gamma \left( k\right) \Gamma \left( n+1\right) 12^{n}}%
\frac{b_{k}}{\left( m+\frac{11}{12}\right) ^{n+k}} \\
&=&\sum_{s=0}^{\infty }\frac{d_{s}}{\left( m+\frac{11}{12}\right) ^{s}},
\end{eqnarray*}%
where%
\begin{equation*}
d_{s}=\sum_{k+n=s}\left( -1\right) ^{n}\frac{\Gamma \left( k+n\right) }{%
\Gamma \left( k\right) \Gamma \left( n+1\right) 12^{n}}b_{k}.
\end{equation*}%
Now the conclusion follows by identifying the coefficients in%
\begin{eqnarray*}
\left( 1+\frac{1}{m}\right) ^{m} &=&e\left( b_{0}-\frac{b_{1}}{m+1}-\frac{%
b_{2}}{\left( m+1\right) ^{2}}-\cdots \right) \\
&=&e\left( d_{0}-\frac{d_{1}}{m+\frac{11}{12}}-\frac{d_{2}}{\left( m+\frac{11%
}{12}\right) ^{2}}-\cdots \right) .
\end{eqnarray*}
\end{proof}

We concentrate now to give a recurrence relation for sequence $d_{n}.$ First
we state the following

\begin{lemma}
Let%
\begin{equation*}
g\left( t\right) =\left( \frac{1-\frac{11}{12}t}{1+\frac{1}{12}t}\right) ^{%
\frac{11}{12}-\frac{1}{t}},\ \ \ 0<t<1.
\end{equation*}%
Then%
\begin{equation}
g\left( t\right) =e\left( c_{0}+c_{1}t+c_{2}t^{2}+...\right) ,  \label{dd}
\end{equation}%
where $c_{0}=1$ and%
\begin{equation*}
c_{n}=\frac{1}{n}\sum_{k=0}^{n-1}a_{n-k-1}c_{k},
\end{equation*}%
where%
\begin{equation*}
a_{n}=\frac{n+1}{12^{n+2}}\left( \frac{\left( -1\right) ^{n+1}11-11^{n+2}}{%
n+1}-\frac{\left( -1\right) ^{n}-11^{n+2}}{n+2}\right) .
\end{equation*}
\end{lemma}

\begin{proof}
Using Maclaurin series, we have%
\begin{equation*}
\ln g\left( t\right) =1+\sum_{n=1}^{\infty }\frac{1}{12^{n+1}}\left( \frac{%
\left( -1\right) ^{n}11-11^{n+1}}{n}-\frac{\left( -1\right) ^{n-1}-11^{n+1}}{%
n+1}\right) t^{n},
\end{equation*}%
thus%
\begin{equation*}
\frac{g^{\prime }\left( t\right) }{g\left( t\right) }=\sum_{n=0}^{\infty }%
\frac{n+1}{12^{n+2}}\left( \frac{\left( -1\right) ^{n+1}11-11^{n+2}}{n+1}-%
\frac{\left( -1\right) ^{n}-11^{n+2}}{n+2}\right) t^{n}.
\end{equation*}%
Now we can denote $g^{\prime }\left( t\right) =g\left( t\right) \varphi
\left( t\right) ,$ where%
\begin{equation*}
\varphi \left( t\right) =\sum_{n=0}^{\infty }a_{n}t^{n},\ \ \ a_{n}=\frac{n+1%
}{12^{n+2}}\left( \frac{\left( -1\right) ^{n+1}11-11^{n+2}}{n+1}-\frac{%
\left( -1\right) ^{n}-11^{n+2}}{n+2}\right) .
\end{equation*}%
Thanks to Leibniz rule,%
\begin{equation*}
g^{\left( n\right) }\left( t\right) =\sum_{k=0}^{n-1}\left( 
\begin{array}{c}
n-1 \\ 
k%
\end{array}%
\right) g^{\left( n-k-1\right) }\left( t\right) \varphi ^{\left( k\right)
}\left( t\right) ,
\end{equation*}%
but $\varphi ^{\left( k\right) }\left( 0\right) =k!a_{k},$ so%
\begin{equation*}
g^{\left( n\right) }\left( 0\right) =\sum_{k=0}^{n-1}\frac{\left( n-1\right)
!}{\left( n-k-1\right) !}a_{k}g^{\left( n-k-1\right) }\left( 0\right) .
\end{equation*}%
Now%
\begin{eqnarray*}
g\left( t\right) &=&\sum_{n=0}^{\infty }\frac{g^{\left( n\right) }\left(
0\right) }{n!}t^{n}=\sum_{n=0}^{\infty }\left( \sum_{k=0}^{n-1}\frac{\left(
n-1\right) !}{\left( n-k-1\right) !}a_{k}g^{\left( n-k-1\right) }\left(
0\right) \right) \frac{t^{n}}{n!} \\
&=&\sum_{n=0}^{\infty }\left( \frac{1}{n}\sum_{k=0}^{n-1}\frac{g^{\left(
k\right) }\left( 0\right) }{k!}a_{n-k-1}\right) t^{n} \\
&=&e\sum_{n=0}^{\infty }c_{n}t^{n},
\end{eqnarray*}%
where%
\begin{equation*}
c_{n}=\frac{1}{ne}\sum_{k=0}^{n-1}\frac{g^{\left( k\right) }\left( 0\right) 
}{k!}a_{n-k-1}
\end{equation*}%
As%
\begin{equation*}
c_{n}=\frac{g^{\left( n\right) }\left( 0\right) }{n!e}
\end{equation*}%
we get%
\begin{equation*}
g^{\left( n\right) }\left( 0\right) =\left( n-1\right) !\sum_{k=0}^{n-1}%
\frac{g^{\left( k\right) }\left( 0\right) }{k!}a_{n-k-1},
\end{equation*}%
or%
\begin{equation*}
c_{n}=\frac{1}{n}\sum_{k=0}^{n-1}c_{k}a_{n-k-1},
\end{equation*}%
which is the conclusion.
\end{proof}

By taking $t=\left( m+\frac{11}{12}\right) ^{-1}$ in (\ref{dd}), we obtain
the following

\begin{theorem}
The following representation holds true%
\begin{equation*}
\left( 1+\frac{1}{m}\right) ^{m}=e\left( 1-\frac{d_{1}}{m+\frac{11}{12}}-%
\frac{d_{2}}{\left( m+\frac{11}{12}\right) ^{2}}-\frac{d_{3}}{\left( m+\frac{%
11}{12}\right) ^{3}}-\cdots \right) ,
\end{equation*}%
and%
\begin{equation*}
d_{n}=\frac{1}{n}\sum_{k=0}^{n-1}\frac{n-k}{12^{n-k+1}}\left( \frac{\left(
-1\right) ^{n-k}11-11^{n-k+1}}{n-k}-\frac{\left( -1\right)
^{n-k-1}-11^{n-k+1}}{n-k+1}\right) d_{k}
\end{equation*}%
(here $d_{0}=-1$).
\end{theorem}

In the last part of this paper we give an integral representation of $d_{n}.$
To do this, we make appeal to the following result stated in \cite{alzer}.

\begin{lemma}
Let%
\begin{equation*}
h(x)=(x+1)\left[ e-\left( 1+\frac{1}{x}\right) ^{x}\right] \qquad (x>0).
\end{equation*}%
Then we have%
\begin{equation*}
h(x)=\frac{e}{2}+\frac{1}{\pi }\int_{0}^{1}\frac{s^{s}(1-s)^{1-s}\sin (\pi {s%
})}{x+s}ds.
\end{equation*}
\end{lemma}

\begin{theorem}
Let $d_{n}$ be the sequence defined by (\ref{D}), and let%
\begin{equation*}
g(s)=\frac{1}{\pi }{s^{s}}(1-s)^{1-s}sin(\pi {s}).
\end{equation*}%
Then%
\begin{equation*}
d_{n}=\frac{\left( -1\right) ^{n}}{12^{n-1}}\left( -\frac{1}{2}+\frac{1}{e}%
\int_{0}^{1}\frac{\left( 12s-11\right) ^{n-1}-1}{s-1}g\left( s\right)
ds\right) ,\ \ \ n=2,3,\ldots .
\end{equation*}
\end{theorem}

\begin{proof}
By Lemma 3 , we have%
\begin{equation*}
e-\left( 1+\frac{1}{x}\right) ^{x}=\frac{e}{2}\frac{1}{1+x}+\int_{0}^{1}%
\frac{g\left( s\right) }{\left( x+1\right) \left( x+s\right) }ds.
\end{equation*}%
Thus, from Theorem 2 , we have 
\begin{equation*}
e\left( \frac{d_{1}}{x+\frac{11}{12}}+\frac{d_{2}}{\left( x+\frac{11}{12}%
\right) ^{2}}+...\right) =\frac{e}{2}\frac{1}{1+x}+\int_{0}^{1}\frac{g\left(
s\right) }{\left( x+1\right) \left( x+s\right) }ds.
\end{equation*}%
With $t=\left( x+\frac{11}{12}\right) ^{-1},$ we obtain%
\begin{equation*}
e\left( d_{1}t+d_{2}t^{2}+...\right) =\frac{e}{2}\frac{12t}{12+t}%
+\int_{0}^{1}\frac{144t^{2}}{12+t}\frac{g\left( s\right) }{12+12st-11t}ds.
\end{equation*}%
Differentiation gives%
\begin{equation*}
d_{n}=\frac{\left( -1\right) ^{n}}{12^{n-1}}\left( -\frac{1}{2}+\frac{1}{e}%
\int_{0}^{1}\frac{\left( 12s-11\right) ^{n-1}-1}{s-1}g\left( s\right)
ds\right) ,\ \ \ n=2,3,...
\end{equation*}%
This completes the proof of Theorem 3.
\end{proof}


\begin{thebibliography}{9}
\bibitem{alzer} H. Alzer, C. Berg, Some classes of completely monotonic
functions, Ann. Acad. Sci. Fennicae, 27 (2002), 445-460.

\bibitem{bicheng} Y. Bicheng and L. Debnath, Some inequalities involving the
constant $e$ and an application to Carleman's inequlity. J. Math. Anal.
Appl., 223 (1998), 347--353.

\bibitem{gy} M. Gyllenberg, P. Yan, On a conjecture by Yang, J. Math. Anal.
Appl. 264 (2001) 687--690.

\bibitem{carl} G.H. Hardy, J.E. Littlewood and G. Polya, Inequalities,
Cambridge Univ. Press, London, 1952.

\bibitem{m1} C. Mortici, Refinements of some bounds related to the constant $%
e,$ Miskolc Math. Notes, (2011), in press.

\bibitem{ping} Y. Ping and S. Guozheng, A Strengthened Carleman's
inequality. J. Math. Anal. Appl., 240 (1999), 290--293.

\bibitem{yang} X.Yang, On Carleman's inequality, J. Math. Anal. Appl., 253
(2001), 691--694.

\bibitem{yang2} X. Yang, Approximations for constant $e$ and their
applications, J. Math. Anal. Appl., 262 (2001), 651--659

\bibitem{yue} H. Yue, A Strengthened Carleman's Inequality, Comm. Math.
Anal., 1 (2006), no. 2, 115-119.
\end{thebibliography}
\end{document}